\documentclass{amsart}
\newcommand{\Q}{\mathbb{Q}}
\newcommand{\Qp}{\Q_p}
\newcommand{\Qpb}{\bar{\Q}_p}
\newcommand{\dr}{\textup{dR}}
\newcommand{\hdr}{H_{\dr}}
\newcommand{\DD}{\mathcal{D}}
\newcommand{\Eg}{E_{\textup{good}}}

\newcommand{\ptt}{\partial_t}

\newcommand{\OO}{\mathcal{O}}
\newcommand{\tang}{\mathcal{T}}
\newcommand{\cano}{\mathbf{\omega}_C}
\newcommand{\invv}{^{\otimes -1}}
\newcommand{\ocol}{\OO_{\textup{Col}}}
\newcommand{\LL}{\mathcal{L}}
\newcommand{\odi}{\OO(\Delta)}
\newcommand{\odii}{\OO(-\Delta)}
\newcommand{\lgod}{\log_{\odi}}
\newcommand{\lgodi}{\log_{\odii}}
\newcommand{\loti}{\log_{\tang}}
\newcommand{\hoti}{h_{\tang}}
\newcommand{\delbar}{\bar{\partial}}

\newcommand{\bom}{\bar{\omega}}
\newcommand{\ord}{\operatorname{ord}}
\newcommand{\Div}{\operatorname{Div}}
\newcommand{\tr}{\operatorname{tr}}
\newcommand{\pab}{(A,B)}

\newtheorem{theorem}{Theorem}[section]

\newtheorem{proposition}[theorem]{Proposition}
\newtheorem{lemma}[theorem]{Lemma}
\newtheorem{corollary}[theorem]{Corollary}
\theoremstyle{definition}
\newtheorem{definition}[theorem]{Definition}
\newtheorem{remark}[theorem]{Remark}
\newtheorem{example}[theorem]{Example}
\newtheorem{assumption}[theorem]{Assumption}
\numberwithin{equation}{section}

\begin{document}
\title{Coleman-Gross height pairings and the $p$-adic sigma function}
\author{Jennifer S. Balakrishnan}
\address{Department of Mathematics\\ Harvard University\\ One Oxford Street\\
Cambridge, MA 02138}
\author{Amnon Besser}
\address{School Of Mathematical And Statistical Sciences\\
Arizona State University\\
PO Box 871804, Tempe, AZ, 85287-1804\\
USA
}
\curraddr{
Department of Mathematics\\
Ben-Gurion University of the Negev\\
P.O.B. 653\\
Be'er-Sheva 84105\\
Israel
}

\begin{abstract}
We give a direct proof that the Mazur-Tate and Coleman-Gross heights
on elliptic curves coincide. The main ingredient is to extend the
Coleman-Gross height to the case of divisors with non-disjoint
support and, doing some $p$-adic analysis, show that, in particular,
its component above $p$ gives, in the special case of an ordinary
elliptic curve, the $p$-adic sigma function. 

We use this result to
give a short proof of a theorem of Kim characterizing integral points
on elliptic curves in some cases under weaker assumptions. As a further application, we give new formulas to compute double Coleman integrals from tangential basepoints.
\end{abstract}
\maketitle

\section{Introduction}
\label{sec:intro}

Let $p$ be a prime number, let $F$ be a number field and  let $E$ be
an elliptic curve over $F$. In this paper we relate two constructions
of $p$-adic heights for $E$.

Suppose $E$ has ordinary
reduction at all primes above $p$. Then, the construction of Mazur and
Tate~\cite{Maz-Tat83,MTT86,Maz-Tat91,Maz-Ste-Tat05} uses
the $p$-adic sigma function constructed in~\cite{Maz-Tat91} to define the
$p$-adic height of a point, the resulting function being quadratic
resulting in a height pairing.

On the other hand, Coleman and
Gross~\cite{Col-Gro89} defined for any curve of positive genus with
good reduction above $p$ a
$p$-adic height pairing for divisors of degree $0$ with disjoint
support. This factors via the Jacobian of the curve and thus extends
to any pair of divisors without the disjoint support assumption.

Thus, there is a certain range in which both constructions
apply and then they are known to be equal by either the work of
Coleman~\cite{Col91} tying the Coleman-Gross height pairing with the
$p$-adic height defined via bi-extensions as in~\cite{Maz-Tat83}, or
by the work of the second author~\cite{Bes99a} tying it with the general height
pairing defined by Nekov{\'a}{\v r}~\cite{Nek93}. This proof is ultimately
global in nature, even though both heights are sums of local terms.

Both
constructions
have subsequently been implemented numerically: the Mazur-Tate
construction by Mazur, Stein, and Tate~\cite{Maz-Ste-Tat05} and the Coleman-Gross
construction, in the case of hyperelliptic curves, by the authors~\cite{Bes-Bal10}.

It is thus interesting to prove the equality of these two height
pairings by comparing the local terms. Note though that a priori the
local terms are defined for two divisors with disjoint support in the
Coleman-Gross construction and for a divisor with itself in the
Mazur-Tate construction. Our modest goal is to provide such a comparison.

What we in fact do is to stretch both definitions somewhat, and we
compare them on the extended range. From the Coleman-Gross side, we
discuss the extension of the height pairing to divisors with
non-disjoint support, which is suggested by the work of Gross in the
complex case~\cite{Gross86}. This involves choosing a tangent vector at
each of the points in the common support of the divisors. The height
pairing decomposes as a sum of local terms, which depend on the
tangent vectors but their sum does not. The terms away from $p$ are
already discussed by Gross and the extension to places above $p$ is
rather straightforward. We note that the procedure is rather easy to implement
numerically, and there may well be some interest in such an implementation.

There is little trouble in comparing the height pairing away from $p$,
so the interest lies in the $p$-adic theory. We analyze it using tools
from $p$-adic Arakelov theory~\cite{Bes00}. For a curve $C$ over the
algebraic closure of the $p$-adic numbers, fixing a point $x_0$ and a
tangent vector $t_0$ at $x_0$, the local height $h(x-x_0,x-x_0)$
becomes a function on the tangent bundle to $C$. We analyze this
function for a general curve. For an elliptic curve, the choice of an
invariant differential provides a canonical choice of tangent
vectors, and we prove that the local height, computed using these
choices, is essentially the logarithm of the $p$-adic sigma
function. From this it is standard to deduce the equality of the
heights.

As a corollary, we give a short proof for a recent theorem of
Kim~\cite{Kim10} (corrected by the first author, Kedlaya and Kim in~\cite{BKK11}),
giving a $p$-adic characterization of
integral points of some elliptic curves over $\Q$. In particular, we
remove the assumption there about the finiteness of the $p$-part of
the Tate-Shafarevich group.

We conclude by giving some numerical examples illustrating these results, as well as new formulas to compute double Coleman integrals from tangential basepoints.

We would like to thank J. Ellenberg for providing the motivation for this work.
\section{An extension of the Coleman-Gross height to non-disjoint
  support} 
\label{sec:cg}

In this section we recall what we need about the Coleman-Gross height
pairing~\cite{Col-Gro89} and we make the easy extension to divisors
with non-disjoint
support suggested in the complex case by~\cite{Gross86}.

The height pairing is defined for divisors of degree $0$ on a smooth
complete curve
$X/F$ defined over a number field $F$.
\begin{remark}
  Since we will be using, as in~\cite{Bes00}, the version of Coleman
  integration developed by Vologodsky~\cite{Vol01}, we do not assume
  that $X$ has good reduction at primes above $p$. However, the
  resulting height pairing is only assured to coincide with standard
  definitions of height pairings under this additional
  assumption~\cite{Col91,Bes99a}.
\end{remark}

To define the height pairing
one needs the following data~\cite[Section~2]{Bes99a}:
\begin{itemize}
\item  A ``global log''- a continuous idele class character
 \begin{equation*}
   \ell: \mathbb{A}_K^\ast/K^\ast \to \Q_p
 \end{equation*}
 such that above $p$ the local components are ramified

\item For each $v|p$ a choice of a subspace $W_v\subset\hdr^1(X\otimes
  K_v/K_v)$ complementary to the space of holomorphic forms.
\end{itemize}
To use $p$-adic Arakelov geometry we need to impose the following
\begin{assumption}\label{isotrop}
  Each subspace $W_v$ is isotropic with respect to the cup product.
\end{assumption}
This assumption guarantees that the height pairing is
symmetric~\cite[Proposition~5.2]{Col-Gro89}. It is
also automatically satisfied for elliptic curves, as the dimension of
$W_v$ is $1$.

One obtains from $\ell$, for each $v|p$, $\Qp$-linear maps $\tr_v:K\to
\Qp$
as well as a choice of the $p$-adic logarithm $ \log_v : K_v^\ast
\to K_v$ such that $\ell_v = \tr_v\circ \log_v $.

Suppose now that $D_1$ and $D_2$ are $F$-rational divisors on $X$ of
degree $0$ and make the additional assumption that
\begin{equation*}
  D_1 = \sum_i n_i (x_i)\;,\text{ with $x_i$ rational over $F$}.
\end{equation*}
This assumption is eventually removed by going to a field extension
where $D_1$ has this property and using the functoriality properties
of the height. We do not assume that $D_1$ and $D_2$ have disjoint
support. The price we have to pay for that is that the height pairing
decomposes as usual
$$h({D_1},{D_2})= \sum_v h_v({D_1},{D_2})$$
over all finite places $v$, but the local components depend in
addition on the choice of $F$-rational tangent vectors at the points
of the common support. As we will see, the behavior of the local heights, with
respect to these vectors, is such that the global
height is independent of these choices.

The local component $h_v$ depends only on the completion at $v$, so let
$K=F_v$ and $C=X_v=X\otimes_F K$. Let $\chi=\ell_v$ be the local component
of $\ell$ at $v$. When $v|p$ let $\log=\log_v$, and $\tr=\tr_v$ be the
branch of the 
$p$-adic logarithm and the trace map deduced from it, and let $W=W_v$
be the complementary subspace.

To compute the local height pairings we need tangent vectors. For ease
of presentation, let us
assume that we have chosen, at each $K$-rational point $x$ of
$C$,  a
tangent vector $t=t_x$. For each such $x$, we further choose a local
parameter $z=z_x$ normalized in such a way that $\ptt z = 1$, where
$\ptt$ is the derivation associated to $t$. Using
this parameter, we can define a ``value'' for any rational function
$f$, defined over $K$,
at $x$ by 
\begin{equation*}
  f[x] = f(x,t) = \frac{f}{z^{m}}(x),
\end{equation*}
where $m=\ord_x(f)$ is the order of $f$ at $x$. Note that this clearly
depends only on $t$ and not on the particular choice of parameter
$z$. Given a divisor $D=\sum n_i (x_i)$, with all $x_i$ defined over
$K$, we may define the
value of $f$ at $D$ by the rule
\begin{equation*}
  f[D] := \prod_i f(x_i)^{n_i}\;.
\end{equation*}
We let $ Z^0(C)\subset \Div^0(C) $ be the group of degree zero
divisors on $C$ and its subgroup generated by $K$-rational points.
\begin{definition}\label{intersect}
  A \emph{local height pairing} on $C$,
  \begin{equation*}
    h: \Div^0(C) \times Z^0(C) \to \Qp
  \end{equation*}
  is a function (more precisely, a collection of functions, one for
  each choice of tangent vectors) satisfying the following conditions:
  \begin{enumerate}
  \item It is bi-additive, continuous and symmetric for divisors with
    disjoint support in $Z^0(C)$.
  \item  \label{second-cond} It satisfies
    \begin{equation*}
      \langle (f),  D_2\rangle = \chi(f[D_2])
    \end{equation*}
    for $f \in k(C)^*$.
  \item \label{third-cond}The dependence on the choice of tangent
    vectors is as follows: if $h'$ is the height function obtained
    from $h$ by changing $t_x$ to $\alpha t_x$, then
    \begin{equation*}
      h'(D_1,D_2) = h(D_1,D_2) + \chi(\alpha)\cdot \ord_xD_1 \cdot
      \ord_xD_2\;.
    \end{equation*}
  \end{enumerate}
\end{definition}
\begin{proposition} If $v$ does not divide $p$, then there
  exists a unique local height pairing on $C$.
\end{proposition}
\begin{proof}
For divisors with disjoint support, see \cite[Prop
1.2]{Col-Gro89} for the uniqueness. Condition~\eqref{second-cond} then clearly
extends the definition uniquely to all divisors. The local height
pairing for divisors with disjoint support can be
constructed~\cite[(1.3)]{Col-Gro89} as 
\begin{equation}
  \label{locCG}
  h({D_1},{D_2}) = \ell_v(\pi_v) \cdot ({D_1},{D_2})\;.
\end{equation}
Here, $({D_1},{D_2})$ denotes intersection multiplicity on a regular model of
$C$ over $\OO_K$ of extensions of ${D_1}$ and ${D_2}$ to this model. To make
this have the required properties one of these extensions has to have
zero intersection with all components of the special fiber. The
same formula can be used for divisors with non-disjoint support
provided that all the tangent vectors are chosen to be
integral~\cite[Section~5]{Gross86} and then extended to arbitrary tangent
vectors using part~\eqref{third-cond} of Definition~\ref{intersect}.
\end{proof}

We now turn to the case $v|p$, where a local height pairing is not
unique, but is constructed (when the support is disjoint)
in~\cite{Col-Gro89}. In this section we only describe the
local height pairing for divisors with disjoint support, deferring
treatment of the general case to Section~\ref{sec:loccol}.

Coleman and Gross describe their local height pairing in
terms of Coleman integration of differentials of the third kind. We
will need, however, the description provided in~\cite{Bes00}. To the
complementary subspace $W$ one associates a $p$-adic Green function,
which is a certain Coleman function on $C\times C-\Delta$, where
$\Delta$ is the diagonal. This Green function is defined only up to an
additive constant. However, for each divisor 
$$D= \sum m_j (y_j)$$
 of degree $0$ on $C$
the function 
$$G_D(x) = \sum_j m_j G(y_j,x)$$
is defined without ambiguity outside the support of $D$. It coincides
with the Coleman integral
used to define the height pairing in Coleman and Gross's
work~\cite[Theorem~7.3]{Bes00}. In particular, if $f$ is a rational
function on $C$ with divisor $(f)$, then
\begin{equation}\label{goff}
  G_{(f)} = \log(f) + \text{ a constant.} 
\end{equation}

Note that with this definition, $D$ should a priori be in $Z^0(C)$, but
we may define $G_D$ for any $D\in \Div^0(C)$ by extending scalars.
If $D_1$ and $D_2=\sum_i n_i (x_i)$ have disjoint support, we may define
\begin{equation*}
  G_{D_1}(D_2)= \sum_i n_i G_{D_1}(x_i),
\end{equation*}
and the local height pairing for such divisors is given by
\begin{equation*}
  h(D_1,D_2)=\tr(G_{D_1}(D_2)).
\end{equation*}

The local height pairing may be extended to divisors with non-disjoint
support in the sense of 
Definition~\ref{intersect}. Explicit formulas for this extension will
be given in Section~\ref{sec:loccol}.

Going back to the global situation, for any finite place $v$ we have,
by change of base, $F_v$-rational divisors $D_1,D_2$ of degree
$0$ on $X_v$, such that $D_1= \sum n_i (x_i)$ with $x_i$
$F_v$-rational. The local height $h_v=h$ on $X_v$ depends on choices
of tangent vectors, but condition~\eqref{third-cond} in
Definition~\ref{intersect} implies that they are only required at the
points of the common support of $D_1$ and $D_2$, and there we have
made a global choice of a tangent vector. We can therefore define
\begin{equation*}
  h(D_1,D_2) = \sum_v h_v(D_1,D_2)\;.
\end{equation*}
Condition~\eqref{third-cond} in
Definition~\ref{intersect} and the fact that $\ell$ is an idele class
character imply now that $h$ is independent of the choices of tangent
vectors and it furthermore factors via the Jacobian of $X$.

\section{The local height as a Coleman function}
\label{sec:loccol}

In this section we deal with the $p$-adic analysis of the local height
pairing at primes above $p$. We will be using some facts from $p$-adic
Arakelov theory, as developed in~\cite{Bes00}. We therefore assume now that $C$ is a complete curve over the algebraic closure $\Qpb$ of $\Qp$.

We will be using the theory of higher Coleman integration, as developed by
Vologodsky~\cite{Vol01}, and discussed further in~\cite{Bes00}. We
will recall necessary facts when needed. For now, suffice it to say
that the theory, dependent on the choice of a branch of the $p$-adic
logarithm, provides a ring $\ocol(U)$ of so-called Coleman functions
on any smooth algebraic variety $U/\Qpb$ containing the ring of
regular functions $\OO(U)$ on $U$,
whose members are locally analytic, and that the differential induces
a short exact sequence
\begin{equation*}
  0\to \Qpb \to \ocol(U) \to (\ocol(U)\otimes_{\OO(U)}
  \Omega^1(U))^{d=0} \to 0
\end{equation*}
(in
this section differential forms and de Rham cohomologies are over $\Qpb$).
Finally, there is a compatible action of $\operatorname{Gal}(\Qpb/\Qp)
$ on this short exact sequence.

We begin by recalling the definition of the constant term of a Coleman
function~\cite{Bes-deJ02}, \cite[Definition~3.1]{Bes-Fur03}.
If $H$ is a Coleman function on $C$, then, locally, near
a point $x\in C$ with local parameter $z$, $H$ can be written as a
polynomial $\sum_{i\ge 0} h_i \log(z)^i$, with $h_i$ meromorphic at
$x$. 
\begin{definition}
The constant term in a Laurent expansion of $h_0$ with respect
to $z$ is called the constant term of $H$ with respect to the parameter
$z$, denoted $c_z(H)$.
\end{definition}

\begin{lemma}\label{logbehav}
  Suppose $dF$ is meromorphic in a neighborhood of $x$ with a simple
  pole at $x$ with residue $r$. Let $t$ be a tangent vector at $x$. Let
  $z$ and $z'$ be two local
  parameters at $x$ with the relation $\ptt z' = \beta \ptt z$. Then
  we have
  \begin{equation*}
    c_z(H)=c_{z'} (H)+r\log(\beta)\;.
  \end{equation*}
  In particular, the \emph{constant term with respect to $t$},
  \begin{equation*}
    H(x,t) := c_z(H)\;,\text{ with $z$ such that $\ptt z =1 $}
  \end{equation*}
  is well-defined, independent of the particular choice of $z$, and in
  addition we have, for any $\alpha \ne 0$
  \begin{equation*}
    H(x,\alpha t) = H(x,t) + r \log(\alpha).
  \end{equation*}
\end{lemma}
\begin{proof}
For a local parameter $z$,  we have near $x$, $H=h+r \log(z)$, with
$h$ analytic, hence $c_z(H)= h(x)$. We have 
$$z'/z = \beta \cdot (1+g(z))$$
 with $g$ analytic vanishing at $x$.
Thus we have
\begin{equation*}
  H=h'+ r \log(z') = h'+ r(\log(z) + \log(\beta) + \log(1+g(z)))\;.
\end{equation*}
Since $\log(1+g(z)) $ is analytic near $x$ and vanishes at $x$, we find
$h=h'+r\log(\beta)+ r\log(1+g(z))$ and $h(x)=h'(x)+r\log(\beta)$, from
which the result is clear.
\end{proof}

For any $x\in C$ consider the function $G_x(y) := G(x,y)$ on
$C$.
\begin{lemma}\label{greenbehav}
  The form $dG_x$ is meromorphic near $x$, with a simple pole of residue $1$. In
  other words, it satisfies the condition of
Lemma~\ref{logbehav} with $r=1$.
\end{lemma}
\begin{proof}
This is essentially obvious from the theory in~\cite{Bes00} but is not
stated there explicitly. From the definition of $G$ it follows that
$G_x= \log(1)$, with $1$ being the canonical section of $\OO(x)$ (for
the terminology of log functions see below). If
$z$ is a local parameter at $x$, then $z^{-1} \cdot 1$ extends to
a non-vanishing
section of $\OO(x)$ near $x$. From the properties of log functions it
follows that $G_x-\log(z)$ extends to an analytic function near $x$,
proving the result.
\end{proof}

At this point we can construct the extension of the Coleman-Gross
local height pairing to divisors with non-disjoint support.
\begin{proposition}
The formulas, for  $D_1$ and $D_2=\sum_i n_i (x_i)$,
\begin{equation*}
  G_{D_1}[D_2]= \sum_i n_i G_{D_1}(x_i,t_{x_i})
\end{equation*}
and
\begin{equation*}
  h(D_1,D_2)=\tr(G_{D_1}[D_2]),
\end{equation*}
define a local height pairing in the sense of
Definition~\ref{intersect}.
\end{proposition}
\begin{proof}
Everything is essentially clear. One just needs to observe that as a
consequence of~\eqref{goff} we have $G_{(f)}[D_2]=\log (f[D_2])$, and
that condition~\eqref{third-cond} in Definition~\ref{intersect} is an immediate consequence of
Lemma~\ref{greenbehav} and Lemma~\ref{logbehav}.
\end{proof}

We now turn to a more detailed analysis of the local height pairing.
\begin{definition}\label{lotidef}
  For a point $x\in C$ and a tangent vector $t$ at $x$ define
  \begin{equation}
    \loti(x,v)=
    G_x(x,v).
  \end{equation}
\end{definition}
Lemmas \ref{greenbehav}~and~\ref{logbehav} imply that $\loti$
 is a quasi-log function, in the sense of~\cite[Definition~4.1]{Bes00},
on the tangent line bundle $\tang$ on $C$. Recall that a quasi-log function on
a line bundle is a function on the total space minus the zero section
which on every fiber is the $p$-adic logarithm (as chosen) plus a
constant. A quasi-log
function is a log-function if it is furthermore a Coleman function
with some additional properties. In~\cite[Theorem~5.10]{Bes00} we constructed a
log-function $\lgod$, unique up to an additive constant, on the line
bundle $\odi$, with certain properties. The Green function $G$ is then
\begin{equation}
  \label{Gdef}
  G = \lgod(1) \text{ on } C\times C-\Delta,
\end{equation}
where $1$ is  the canonical section of $\odi$  on $C\times C-\Delta$.

For a line bundle $\LL$ let $\LL\invv:=\operatorname{Hom}(\LL,\OO)$ be
the inverse line
bundle. Given a log function on $\LL$ there is an induced one on
$\LL\invv$, which can be described by the condition that the sum of
the logs of two dual vectors
(two vectors pairing to $1$) is $0$.

Recall that $\odii|_\Delta$ may be identified with the canonical
bundle $\cano$ of $C$ as follows: let 
$z$ be a local parameter at the neighborhood $U$ of some point $x_0$ on
$C$. The function
$z(x)-z(y)$ has a simple zero on $\Delta$ at $U\times U$, hence defines
a non-vanishing section of $\odii$. Its restriction to $\Delta$ gets
mapped to the section $dz$ of $\cano$. By duality, $\odi|_\Delta$ is
identified with $\tang$.
\begin{proposition}
  The restriction of $\lgod$ to $\Delta$ is equal, via the above
  identification, to the quasi-log function $\loti$ defined
  in~\eqref{lotidef}, which is therefore a log function.
\end{proposition}
\begin{proof}
Fix a point $x_0\in C$ and a tangent vector $t$ at $x_0$.
Let $z$ be a local parameter at a point $x_0$ such that $\ptt z =
1$. As above, $z(x)-z(y)$ defines a non-vanishing section of
$\odii$. We find, for $x\ne y$ in a neighborhood of $x_0$,
\begin{align*}
  u(x,y):&=\lgodi(z(x)-z(y)) &\\
         &=\log(z(x)-z(y))+\lgodi(1) &\text{by the behavior of log} \\
         &=\log(z(x)-z(y))-\lgod(1) &\text{since $1\in \odi$ is dual to
           $1\in \odii$}\\
         &=\log(z(x)-z(y))-G(x,y) &\text{by~\eqref{Gdef}.}
\end{align*}
The function $u$ extends to $(x_0,x_0)$, and by the
identification above, its value there is $\log_{\cano}(dz)(x_0)$.
Fixing $y=x_0$ we find
\begin{equation*}
  G_{x_0}(x) = -u(x,x_0)+ \log(z),
\end{equation*}
and by Definition~\ref{lotidef} we find
\begin{equation*}
  \loti(x_0,t) = -u(x_0,x_0)= -\log_{\cano}(dz)(x_0),
\end{equation*}
and the proof is complete upon noting that $dz$ at $x_0$ is dual, by
definition, to $t$.
\end{proof}
We next recall~\cite[Proposition~4.4]{Bes00} that a log function on a
line bundle on a variety $U$ has (not always,
but for example if $U$ is a complete smooth curve with positive genus) a
curvature, which is an element of $\hdr^1(U)\otimes \Omega^1(U)$. It
is defined as the unique such
element whose pullback to the total space of the line bundle minus the
zero section gives the $p$-adic $\delbar$ operator
(\cite[Section~6]{Bes99}~and~\cite[p. 323]{Bes00})
applied to the
differential of the log function (we will recall later the definition
of the $p$-adic $\delbar$ operator in an affine situation when we need
it).
\begin{remark}\label{anysec}
  When $U$ is affine, the curvature is also $\delbar\circ d $ of the
  log of any non-vanishing section.
\end{remark}
 We now identify the curvature for $\loti$. To do so, recall first that
as part of the data for the height pairing we are given a
decomposition $\hdr^1(C) = W\oplus \Omega^1(C)$. If $\{\omega_i,
i=1,\ldots, g\}$  is a basis for
  $\Omega^1(C)$ there exists a unique choice of a basis $\{\bom_i,
i=1,\ldots, g\}$ for $W$ which is dual to it via the cup product,
i.e., $\bom_i \cup \omega_j = \delta_{ij}$. We have

\begin{proposition}
  For any choice of a basis $\{\omega_i, i=1,\ldots, g\}$ as above the
  curvature of $\loti$ is given by
  \begin{equation*}
    curve(\loti) = (2-2g)\mu\;,\text{ with }
    \mu = \frac{1}{g} \sum_{i=1}^g \bom_i \otimes \omega_i\;.
  \end{equation*}
\end{proposition}
\begin{proof}
According to~\cite[Definition~5.1 and Theorem 5.10]{Bes00}, the
curvature of $\lgod$ is
\begin{equation*}
   \pi_1^\ast \mu + \pi_2^\ast \mu
    - \sum_{i=1}^g \left(\pi_1^\ast \bom_i \otimes  \pi_2^\ast\omega_i
      +\pi_2^\ast \bom_i \otimes  \pi_1^\ast\omega_i
    \right),
\end{equation*}
where $\pi_i$ are the two projections from $C\times C$ to $C$.
Since the curvature behaves nicely with respect to restriction, the
result follows immediately by pulling back to the diagonal.
\end{proof}
\begin{remark}
Note that the corollary is compatible
with~\cite[Proposition~4.4]{Bes00} in the
sense that applying the cup product to the curvature one is expected
to get the degree of the line bundle, and applying it to $\mu$ one
gets $1$.
\end{remark}

We next fix $x_0\in C$ and study the local height function
$h_v(x-x_0,x-x_0)$ itself. As
in the introduction we fix a tangent vector $t_0$ at $x_0$ and the
local height function depends on $x$ and on a tangent vector $t$ at
$x$. Is is easily seen to be a quasi-log function on $\tang|_{C-x_0}$,
and we denote it by $\hoti$.
We have
\begin{equation*}
  \hoti(x,t) = G_x(x,t)-2G_{x_0}(x) + G_{x_0}(x_0,t_0)\;.
\end{equation*}
Since $G_{x_0}$ is the log of the section $1$ of $\OO(x_0)$, we can
interpret $\hoti$, ignoring the last summand which is a
constant, as the pullback of the log function on $\tang \otimes
\OO(-2x_0)$ under the isomorphism of the above with $\tang$ on $C-x_0$
provided by the canonical section $1$. We have the following
information about $\hoti$, which we know how to use to fully
characterize it only in the genus $1$ case.
\begin{proposition}\label{gendata}
  The curvature of $\hoti$ is $-2g\mu|_{C-x_0}$. If $t(x)$ is a
  section of $\tang$ near $x_0$ whose value at $x_0$ is $t_0$, then
  the pullback of $\hoti$ under $t$ has constant term $0$ with respect
  to $t_0$.
\end{proposition}
\begin{proof}
The first statement follows because the curvature of $\OO(x_0)$ is
$\mu$ by~\cite[Corollary~5.12 and Definition~5.7]{Bes00}. For the
second statement we observe that
$G_x(x,t(x))$ has value $ G_{x_0}(x_0,t_0)$ at $x_0$, and that this is
by definition also the constant term of $G_{x_0}(x)$ with respect to $t_0$. 
\end{proof}

\section{The local height for elliptic curves}
\label{sec:elliptic}

In this section we specialize the consideration in the previous
section to the case of
an elliptic curve $E$. We fix $x_0$ to be the identity element $0$ of
the group law. We furthermore fix an invariant differential
$\omega$. This provides us by duality a canonical invariant section
$v$ of the
tangent bundle. We may use this to pullback the log function $\hoti$ to
obtain a Coleman function $\tau$ on $E-0$. Let $\bar{\omega}\in W$ be
the dual class.
\begin{theorem}\label{mainB}
  Let $\eta$ be the unique one-form of the second kind on $E$ with a double pole
  at $0$ and no other poles, representing the cohomology class
  $\bar{\omega}$. Then $\tau$ is the unique Coleman integral of the form
  \begin{equation*}
    \tau = -2\int \left(\omega \times \int \eta \right)
  \end{equation*}
  which is in addition a symmetric function and $\tau(0,v)=0$.
\end{theorem}
\begin{proof}
The function $\tau$ is symmetric by functoriality of its construction
with respect to the involution $x\to -x$ of 
$E$, which induces $-1$ on the first cohomology, hence preserves
$W$. We have
$\tau(0,v)=0$ by Proposition~\ref{gendata}.
From this Proposition, and
Remark~\ref{anysec}
it follows that
$\delbar d\tau = -2g\mu = -2\bar{\omega}\otimes \omega$. We now
recall~\cite[Proposition~2.7]{Bes00} that for an affine $U$, the
$\delbar$ operator
sends $\omega_1 \times \int \omega_2 $ to $[\omega_2] \otimes
\omega_1$, where $[\omega_2]$ is the cohomology class of
$\omega_2$ in $\hdr^1(U)$. This means that on the
affine $E-0$ we have $d\tau = -2\omega \times \int
\eta'$ where the class of $\eta'$ in $\hdr^1(E-0/K)$ is exactly the
restriction of $\bar{\omega}$. Clearly this means that
$\eta'=\eta$. This determines $\tau$ up to an addition of a Coleman
integral of a constant multiple of $\omega$. Two functions satisfying
all the requirements therefore differ by such an integral which is in
addition symmetric and vanishes at $0$. It is therefore identically $0$.
\end{proof}
\begin{corollary}
  Suppose that $E$ is ordinary and $W$ is the unit root subspace, then
  $\tau$, restricted to the residue disc of $0$, is just
  $-2\log(\sigma_p)$, where $\sigma_p$ is
  the $p$-adic sigma function, as defined in~\cite{Maz-Tat91}.
\end{corollary}
\begin{proof}
We recall a definition for the $p$-adic sigma function. The
differential $\omega$ determines a Weierstrass model of $E$ over the
ring of integers $\OO_K$,
\begin{equation}\label{weiersras}
  y^2+a_1 xy +a_3 y= x^3+a_2 x^2 + a_4 x+a_6
\end{equation}
in such a way that $\omega$ becomes the differential $dx/(2y+a_1
x+a_3)$ and it determines a parameter at $0$, $t=-x/y$.
We have the following local expansions:
\begin{equation}
\begin{split}
  x&= t^{-2} + \cdots\\
  y&= -t^{-3}+ \cdots \\
  \omega &= dt(1+\cdots )\;.
\end{split}\label{eq:local_exp}
\end{equation}

In particular, the parameter $t$
is normalized with respect to the dual of $\omega$ at $0$. Let $\eta$
be the form described in Theorem~\ref{mainB}, which in particular
satisfies $ [\eta]\cup [\omega] = 1$. Using the
description of cup products in terms of residues and integrals, one
finds the local expansion for $\eta$ is
\begin{equation*}
  \eta = dt(-t^{-2}+\cdots )\;.
\end{equation*}
According
to~\cite[Theorem~1.3]{Maz-Ste-Tat05}, $\sigma$ is the unique
odd function of the form $\sigma=t+\cdots $ satisfying
the differential equation
\begin{equation*}
  x(t)+c = -\frac{d}{\omega}\left(\frac{1}{\sigma}
    \frac{d\sigma}{\omega}\right)
\end{equation*}
with $$c=\frac{a_1^2+4a_2}{12} -\frac{E_2(E,\omega)}{12} .$$
Thus, we can write, in a neighborhood of $0$,
\begin{equation*}
  \log(\sigma) = -\int \left(\omega \times \int \eta'\right) \text{
    with } \eta' =
  (x+c) \omega\;.
\end{equation*}
Clearly, $\eta'$ is a form of the second kind with a double pole at
$0$ and no other poles. Its local expansion is $dt(t^{-2}+\cdots)$. We
claim that  $\eta'$ represents
a cohomology class in $W$ (the unit root subspace). Indeed, by
switching to a standard model $y^2=4x_3- g_2 x-g_3$, noting that in this
model $\omega=dx/y $, letting $\eta_0= xdx/y$ (this is what is usually
denoted $\eta$) and $u$ a generator of $W$, we have,
by~\cite[A.2.4.1]{Kat73a}
\begin{equation*}
  E_2(E,\omega) = 12 \frac{\eta_0 \cup u}{\omega \cup u}
\end{equation*}
so
\begin{equation*}
  \eta' = \eta_0+ c\omega = \eta_0 -  \frac{\eta_0 \cup u}{\omega
    \cup u} \omega = \textup{const}\cdot \left(\omega\cup u)\eta_0
    -(\eta_0 \cup u) \omega \right),
    \end{equation*}
and this is clearly a constant multiple of $u$. Thus,
$\eta'=-\eta$. We therefore find that
$-2\log(\sigma)$ and $\tau$ satisfy the same differential
equation and they therefore differ by an integral of a constant multiple of
$\omega$. The local expansion of $\sigma$ implies that
$\log(\sigma)(0,v)=0$. The same argument showing the uniqueness of
$\tau$ finishes the proof.
\end{proof}

We are now ready to compare the Coleman-Gross construction to the
Mazur-Tate construction. Suppose now that $E$ is an elliptic curve
defined over $\Q$. Let $\ell$ be the idele class character given as follows:
the component $\ell_p$ is the standard branch $\log_p$ of
$p$-adic logarithm,
normalized so that $\log_p(p)=0$. For $q\ne p$ the character $\ell_q$ is
unramified, with $\ell_q(q)=-\log_p(q)$.
With this choice we have the following:
\begin{corollary}\label{heighformul}
Fix a minimal Weierstrass equation for $E$.
  Suppose the point $x \in E(\Q)$ has coordinates $(a/d^2,b/d^3)$,
  with $a,b,d\in \mathbb{Z}$ and
  $d$ prime to both $a$ and $b$, and does not reduce to a singular point in any of
  the bad reduction fibers of $E$. Then the Coleman-Gross height of
  $x$ is given by
  \begin{equation*}
    h((x)-(0),(x)-(0)) = \tau(x)+2\log_p(d)\;.
  \end{equation*}
  In particular, if $E$ is ordinary at $p$ and $W$ is taken to be the
  unit root subspace, then the Coleman-Gross height coincides with the
  Mazur-Tate height.
\end{corollary}
\begin{proof}
We have already identified the local height $h_p$. The formula follows
because the intersection pairing at any other primes is $2\ord_q(d)$
(see for example~\cite[(26)]{Sil88}). To compare with the Mazur-Tate
height, assume first that the point $x$ reduces to $0$ above $p$. In
this case, our formula matches precisely~\cite[(1.1)]{Maz-Ste-Tat05}
\begin{equation*}
  \frac{1}{p} \log_p\left(\frac{\sigma_p(x)}{d}\right)
\end{equation*}
given that their height is given by $- h((x)-(0),(x)-(0)) /2$ and that
their character is $\ell/p$. Since
both heights are quadratic, they are equal everywhere.
\end{proof}

\section{Application to Kim's theorem on integral points on elliptic
  curves}
\label{sec:Kim}

The paper~\cite{Kim10}, as corrected in~\cite{BKK11} contains a
characterization of integral points on elliptic curves, under certain
hypotheses. By removing some of these hypotheses, we prove a more general theorem, based on  the main results of the
previous section. We view this as an indication that the height
pairing has some anabelian source, a fact we hope to expand on in
later work.
\begin{theorem}\label{kims}
  Let $E$ be an elliptic curve over $\Q$, with
  $\operatorname{rank}E(\Q)=1$, given by a minimal Weierstrass
  equation~\eqref{weiersras}. Let $\Eg$ be the set of non-torsion rational
  points $x\in E(\Q)$
  with integral coordinates satisfying the additional condition that
  $x$ meets each bad reduction fiber at a non-singular point. Let
  $\omega=dx/(2y+a_1 x+a_3)$ be an invariant differential on $E$ and let
  $\eta_0=x\omega$ Let
  \begin{equation*}
    \DD(z):= \int \left(\omega \times \int \eta_0 \right)
  \end{equation*}
  be the unique such integral with $\int \eta_0$ and $\DD$ having
  constant term $0$ with respect to 
  the tangent vector dual to $\omega$ at $0$.
  Then the function $$\frac{\DD(z)}{(\int_0^z \omega)^2}$$ is constant on
  $\Eg\subset E(\Q)\subset E(\Qp)$.
\end{theorem}
\begin{proof}
The condition on $\int \eta_0$ makes $\int \eta_0$ anti-symmetric, hence
$\DD$ symmetric, so, by choosing $W$ appropriately we may assume that
$\DD$ is the function $\tau$ from Theorem~\ref{mainB}. By
Corollary~\ref{heighformul}, the function $\DD$
coincides with
the global height $h$ on $\Eg$, so the function which we wish
to
prove is constant is just $h(x)/(\int_0^x \omega)^2$. But this is
obvious as both the numerator and the denominator are quadratic
functions on the rank $1$ group of rational points of $E$.
\end{proof}
\begin{remark}
The above is indeed a generalization of the results of Kim. He makes
the assumption that at all the primes of bad reduction the N\'eron
model has just one component, and it is easy to see that in this case
$\Eg$ is the set of all (non-torsion) integral points. Kim's
formulation is slightly different but clearly equivalent. Thus, we
obtain an extension of Kim's result when there can be more than one
component in the N\'eron model at some primes, and we further remove
the condition about the $p$-part of the Tate-Shafarevich group of $E$
being finite. From the description of the
local height at primes of bad reduction it is further easy to see,
using the description of the local heights at primes of bad
reduction~\cite{Sil88} that given the reduction types at the bad primes,
and given the height at one integral point, one can write down a
finite number of possibilities for $\DD(z)/(\int_0^z \omega)^2$ at the integral points.
\end{remark}

\section{Examples and another application}

In~\cite{BKK11} several numerical examples of Kim's theorem were given. These
call for the computation of iterated Coleman integrals $\int_v^P
\omega\eta$ where $v$ is a tangential base point at $0$. Although
iterated Coleman integrals on elliptic and 
hyperelliptic curves can be explicitly computed
 \cite{balakrishnan:iterated}, integrating directly from a
tangential basepoint is more subtle than integrating from a finite
Weierstrass point 
because one has to be careful with normalization with respect to the
constant term. The examples in~\cite{BKK11} were all
done by using an integral 2-torsion point, or by a manipulation
using two integral points that relies on Kim's theorem, as well as the
fact that the elliptic curve has rank 1.

Using Theorem~\ref{mainB} we can describe an alternative method for computing
such integrals. We follow this with several examples.

In this section we change the notation slightly. For an elliptic curve
$E$ as in~\eqref{weiersras}, we denote by $\omega,\eta$
the forms $dx/(2y+a_1x+a_3)$ and $xdx/(2y+a_1x+a_3)$ respectively, so
that $\omega$ is the same as in Theorem~\ref{kims} while $\eta$ is
what was $\eta_0$ there. We
also use standard notation for iterated integrals $\int_x^y \alpha
\beta$, which is $\int \left(\alpha \times \int \beta \right)$
evaluated at $y$, where both the inner and outer integrals are fixed
to vanish at $x$. This also applies to $x$ replaced by a tangential
base point at $0$. 

Let $T$ be an auxiliary point. Using the
following two formulas from~\cite{BKK11}
\begin{align*}
  \int_v^P \omega \eta &= \int_T^P \omega \eta +  \int_T^P
  \omega \int_v^T \eta+ \int_v^T  \omega \eta\\
  \int_v^T \eta &= \frac{1}{2}  \int_{-T}^T \eta
\end{align*}
we see that the integral $ \int_v^P \omega \eta$ may be computed
using integrals between points and the integral $ \int_v^T \omega
\eta $.

For any fixed $T$ we can use Theorem~\ref{mainB} to compute the integral above
using an extension of the local height computation in~\cite{Bes-Bal10} to
divisors with non-disjoint support, along the lines described in
Sections \ref{sec:cg}~and~\ref{sec:loccol}. However, if $T$ is a torsion point, a more direct
computation is possible. 

\begin{proposition}
  Let $E$, $\omega$ and $\eta$ be as above and let $v$ be the
  tangential base point at $0$ dual to $\omega$. Let $T=\pab\ne 0$ be
  either a $2$- or $3$- torsion point. Then 
  \begin{equation*}
    \int_v^T \omega \eta =
    \begin{cases}
      \frac{1}{4}\log(f'(A) - a_1 B) & \text{if}\;\, 2T=0,\\
       \frac{1}{3}\log(2B+a_1 A + a_3) & \text{if}\;\; 3T=0.
     \end{cases}
   \end{equation*}
 \end{proposition}
 \begin{proof}
Let $n$ be the order of $T$. Consider the divisor
  $D= (T) - (0)$. Then $nD $ is the divisor of a rational function $g$ on $E$
and by Theorem~\ref{mainB} and Definition~\ref{intersect}
we have 
$$
  2 \int_v^T \omega \eta =  \frac{1}{n} \log(g[T]/g[0]),
$$
where the square bracket notation stands for the normalized value with respect
to the chosen parameter at the point. It remains to compute the right
hand side in both cases.

Recall that the normalized local parameter at $0$ is $t=-x/y$. If
$T=\pab$ is 2-torsion, then $2B+a_1 A+ a_3=0$, and a local parameter at
$T$ is $y_0 = y + \frac{1}{2} (a_1x + a_3)$. To normalize it,
we observe that $y_0^2 = f_0(x)$, with $f_0(x)= f(x) + \frac{1}{4}
(a_1x + a_3)^2$ so that
$2y_0 d y_0 = f_0^\prime (x) dx,$
\begin{equation*}
  \omega = \frac{dx}{2y_0} =\frac{dy_0}{f_0^\prime (x)},
\end{equation*}
and evaluating at $T$ we see that the normalized parameter is
\begin{equation*}
  t_T=\frac{y_0}{f_0^\prime (A)} = \frac{y_0}{f'(A)-a_1 B}\;.
\end{equation*}
We can take $g=x-A$. It is easy to see from the local
expansion~\eqref{eq:local_exp}
that $gt^2$ has value $1$ at $0$. At $T$ we have
\begin{equation*}
  \frac{g}{t_T^2} = \frac{(f'(A)-a_1 B)^2 (x-A)}{y_0^2} =
  \frac{(f'(A)-a_1 B)^2 (x-A)}{f_0(x)} ,
\end{equation*}
and the value at $T$ is going to be
\begin{equation*}
   \frac{(f'(A)-a_1 B)^2 }{f_0^\prime(A)} = f'(A)-a_1 B.
\end{equation*}
This proves the case of 2-torsion.

When $T$ is 3-torsion we proceed as follows: write the equation of the
tangent at $T$ as $y=\alpha x+
\beta= \alpha(x-A)+B$. We can then take $g=y-\alpha x -\beta$ as it has a
pole of order $3$ at infinity, and the fact that $T$ is 3-torsion
exactly means that it vanishes to order $3$ at $T$. It will turn
out that $\alpha,\beta$ are not relevant.

We compute the normalized values of $g$ at $0$, which is the value of
$gt^3$ there, which is $1$ using the expansion~\eqref{eq:local_exp}.

At $T$ we have the parameter $x-A$. The
normalization factor is the value at $T$ of 
$$
  \frac{\omega}{d(x-A)} = \frac{\omega}{dx} = \frac{1}{2y+a_1 x+ a_3}
$$
which is $1/(2B+a_1 A + a_3)$.
The normalized value of $g$ at $T$ is thus $(2B+a_1 A + a_3)^3$ times
the value of  $g/(x-A)^3$ at $T$. To get this value, write $g=
c(x-A)^3 + \cdots$ and substitute this into the equation of the curve
\begin{equation*}
  (g+\alpha (x-A)+ B)^2 + a_1 (x-A) (g+\alpha (x-A)+ B) + (a_3+a_1 A)
  (g+\alpha (x-A)+ B) = (x-A)^3 + \cdots
\end{equation*}
Now we solve for the coefficient of $(x-A)^3$. The result is
\begin{equation*}
  2cB + c(a_3+a_1 A) = 1,
\end{equation*}
and so the normalized value is $(2B+a_1 A + a_3)^2$. The final result
is obtained by taking a log and multiplying by the required factors of
$1/3$ and $1/2$.

\end{proof}

\begin{example}[Two-torsion]Consider the elliptic curve $E: y^ 2= x^3 - 16x + 16$
  with minimal model $y^2 + y = x^3 - x$ (Cremona label ``37a1''). Let
  $P = (0,4)$ be an integral point on $E$ (arising from an integral
  point on the minimal model) and let $\omega = \frac{dx}{2y},
  \eta = x\frac{dx}{2y}$ on $E$. Note that this curve has no
  integral Weierstrass points over $\Q$.

  Nevertheless, one can consider Weierstrass points on $E$ over $\Q_p$
  for various $p$.  For example, at $p = 13,$ we have the point $$W
  =(7 + 7 \cdot 13 + 4 \cdot 13^{2} + 7 \cdot 13^{3} + 6 \cdot 13^{4}
  + O(13^5), O(13^{5}))$$ on the short Weierstrass model, which
  corresponds to the point $$W' = (5 + 8 \cdot 13 + 7 \cdot 13^{2} +
  11 \cdot 13^{3} + 4 \cdot 13^{4} + O(13^5), 6 + 6 \cdot 13 + 6 \cdot
  13^{2} + 6 \cdot 13^{3} + 6 \cdot 13^{4} + O(13^{5}))$$ on the
  minimal model.  Using previous methods, we can compute the invariant
  ratio $$\frac{\int_b^P \omega \eta}{(\int_b^P \omega)^2} =
  11 \cdot 13 + 6 \cdot 13^{2} + 7 \cdot 13^{4} + 6 \cdot 13^{5} +
  O(13^6),$$ as well as the auxiliary integrals \begin{align*}\int_W^P \omega
  \eta &= 12 \cdot 13 + 7 \cdot 13^{2} + 11 \cdot 13^{3} + 7 \cdot
  13^{4} + 11 \cdot 13^{5} + O(13^6)\\
   \int_b^P \omega &= 4
  \cdot 13 + 2 \cdot 13^{2} + 2 \cdot 13^{3} + 10 \cdot 13^{4}
  +O(13^6)\end{align*} to deduce that

$$\int_b^W \omega \eta = 13 + 5 \cdot 13^{2} + 8 \cdot 13^{3} + 4 \cdot 13^{4} + 13^{5} + 9 \cdot 13^{6}.$$

Alternatively, using the proposition above, one can directly compute, via the minimal model, that \begin{align*}\int_b^{W'} \omega' \eta' &= \frac{1}{4}\log(f'(x(W')))\\
  &= 13 + 5 \cdot 13^{2} + 8 \cdot 13^{3} + 4 \cdot 13^{4} + 13^{5} +
  9 \cdot 13^{6},\end{align*} where $\omega',\eta'$ denote the
pullbacks of $\omega,\eta$ to the minimal model.
\end{example}

\begin{example}[Two-torsion]Consider the elliptic curve $E: y^2 = x(x-1)(x+9)$ with
  $1$-forms $\omega=\frac{dx}{2y},\eta = x\frac{dx}{2y}$ and the
  points $W_1 = (1,0)$, $W_2 = (0,0)$. $E$ has minimal model $E': y^2
  = x^3 - x^2 - 30x + 72$ (Cremona label ``480f1''), and note that in
  particular $W_1$ and $W_2$ are integral on the minimal model but
  that the Tamagawa numbers of this curve are not 1.

  The prime $p=7$ is good and ordinary.  Letting $\omega',\eta'$
  denote the pullbacks of $\omega,\eta$ , we compute
  $\int_{W_1}^{W_2} \omega' \eta'$ by using the interpretation
  of the height pairing (via a tangential basepoint $b$), we have
  \begin{align*}\int_{W_1}^{W_2} \omega' \eta' &= \int_{b}^{W_2} \omega' \eta' - \int_b^{W_1} \omega' \eta' \\
    &= \frac{1}{4}\log_p\left(\frac{f'(x(W_2))}{f'(x(W_1))}\right)\\
    &=6 \cdot 7 + 3 \cdot 7^{2} + 3 \cdot 7^{3} + 2 \cdot 7^{5} +
    O(7^{6}) .
  \end{align*}

  As a consistency check, we can compare this calculation to that
  described in \cite{balakrishnan:iterated}, where using near-boundary points in
  each Weierstrass disc, we can compute (via an auxiliary
  non-Weierstrass point $Q = (-1,4)$)
  \begin{align*}\int_{W_1}^Q \omega \eta &= 6 \cdot 7 + 5 \cdot 7^{2} + 4 \cdot 7^{3} + 6 \cdot 7^{4} + O(7^{6})\\
    \int_{W_2}^Q \omega \eta &= 2 \cdot 7^{2} + 7^{3} + 6 \cdot
    7^{4} + 5 \cdot 7^{5} + O(7^{6}),\end{align*} and we see

\begin{align*}\int_{W_1}^{W_2} \omega\eta & = \int_{W_1}^Q \omega \eta - \int_{W_2}^Q \omega \eta \\
  &= 6 \cdot 7 + 3 \cdot 7^{2} + 3 \cdot 7^{3} + 2 \cdot 7^{5} +
  O(7^{6}).\end{align*}

\end{example}

\begin{remark}We note that if one is after the value $\int_v^P
\omega\eta$ where $P$ is an arbitrary point, then for the purposes
of computation, it is much faster to use an intermediate
$3$-torsion point $T$ to break up the path rather than using a
$2$-torsion point. This is because computing iterated integrals
from a Weierstrass endpoint requires computations over very highly
ramified extensions $\Q_p(p^{1/d})$, which is quite slow in
existing implementations.
 \end{remark}

\begin{example}[Three-torsion] We again consider $E: y^2 = x^3 - 16x + 16$ (minimal model
37a), this time over $\Q_7$, where it has a $\Q_7$-rational $3$-torsion
point. As before, let $P = (0,4),$ which is integral on the minimal model.
Again, using the method in~\cite{BKK11}, we compute the
invariant ratio $$\frac{\int_b^P \omega\eta}{\left(\int_b^P
    \omega\right)^2} = 7^{-1} + 1 + 3 \cdot 7 + 6 \cdot 7^{2} + 5
\cdot 7^{4} + 6 \cdot 7^{5} + 6 \cdot 7^{6} + O(7^{7}),$$ as well as
the single integral $$\int_b^P \omega = 2 \cdot 7 + 4 \cdot 7^{3} +
5 \cdot 7^{4} + 4 \cdot 7^{5} + 7^{6} + 2 \cdot 7^{7} + 7^{8} + 7^{9}
+ O(7^{10})$$ to deduce the value of
$$\int_b^P \omega \eta = 4 \cdot 7 + 4 \cdot 7^{2} + 7^{4} + 4 \cdot 7^{5} + 7^{6} + 2 \cdot 7^{7} + 5 \cdot 7^{8} + O(7^{9}).$$

Meanwhile, $E$  has the $3$-torsion point $T$ with \begin{align*}x(T) &= 3 + 5 \cdot 7 + 3 \cdot 7^{2} + 3 \cdot 7^{3} + 6 \cdot 7^{4} + 6 \cdot 7^{5} + 6 \cdot 7^{6} + 5 \cdot 7^{7} + 5 \cdot 7^{8} + 5 \cdot 7^{9} + O(7^{10}), \\
  y(T) &=3 + 3 \cdot 7 + 5 \cdot 7^{2} + 4 \cdot 7^{3} + 2 \cdot 7^{5}
  + 2 \cdot 7^{7} + 6 \cdot 7^{8} + 2 \cdot 7^{9} +
  O(7^{10}).\end{align*}

We compute \begin{align*}\int_b^T \eta &= 1 + 2 \cdot 7 + 3 \cdot 7^{2} + 2 \cdot 7^{3} + 3 \cdot 7^{4} + 6 \cdot 7^{5} + 7^{6} + 5 \cdot 7^{7} + 5 \cdot 7^{8} + 4 \cdot 7^{9} + O(7^{10})\\
  \int_T^P \omega  &=2 \cdot 7 + 4 \cdot 7^{3} + 5 \cdot 7^{4} + 4 \cdot 7^{5} + 7^{6} + 2 \cdot 7^{7} + 7^{8} + 7^{9} + O(7^{10})\\
  \int_T^P \omega\eta &= 5 \cdot 7^{2} + 4 \cdot 7^{4} + 7^{5} +
  2 \cdot 7^{6} + 2 \cdot 7^{7} + 5 \cdot 7^{8} + O(7^{9})
\end{align*}
to deduce
\begin{align*}\int_b^T \omega\eta &= \int_b^P\omega\eta - \int_b^T\eta\int_T^P \omega - \int_T^P \omega\eta\\
  &= 2 \cdot 7 + 2 \cdot 7^{2} + 3 \cdot 7^{3} + 6 \cdot 7^{4} + 2
  \cdot 7^{5} + 6 \cdot 7^{6} + 7^{7} + 5 \cdot 7^{8} +
  O(7^9)\end{align*}

Meanwhile, the corresponding $3$-torsion point $T_1$ on the minimal
model has coordinates

\begin{align*}x(T_1) &= 6 + 4 \cdot 7 + 2 \cdot 7^{2} + 4 \cdot 7^{3} + 7^{4} + 5 \cdot 7^{5} + 7^{6} + 3 \cdot 7^{7} + 7^{8} + 3 \cdot 7^{9} + O(7^{10}) \\
  y(T_1) &= 3 \cdot 7 + 5 \cdot 7^{2} + 3 \cdot 7^{3} + 2 \cdot 7^{4}
  + 2 \cdot 7^{5} + 4 \cdot 7^{6} + 4 \cdot 7^{9} +
  O(7^{10}).\end{align*}
and we compute directly using the formula that \begin{align*}\int_b^{T_1}\omega'\eta' &= \frac{1}{3}\log(2y(T_1)+a_1 x(T_1) + a_3)) \\
  &= 2 \cdot 7 + 2 \cdot 7^{2} + 3 \cdot 7^{3} + 6 \cdot 7^{4} + 2
  \cdot 7^{5} + 6 \cdot 7^{6} + 7^{7} + 5 \cdot 7^{8} + 7^{9} +
  O(7^{10})\end{align*}

  \end{example}

\begin{example}[Three-torsion] We consider $E: y^2 = x^3 +405x + 16038$ (minimal model 53a) over $\Q_7$, where it has a $\Q_7$-rational $3$-torsion point. Let $P = (-9,108)$, which is integral on the minimal model.
Using the method in~\cite{BKK11}, we compute the invariant ratio

$$\frac{\int_b^P \omega\eta}{\left(\int_b^P
    \omega\right)^2} = 6 \cdot 7^{-1} + 6 + 7 + 6 \cdot 7^{2} + 4 \cdot 7^{3} + 2 \cdot 7^{4} + 5 \cdot 7^{5} + 4 \cdot 7^{6} + O(7^7)$$
    
    as well as the single integral  $$\int_b^P \omega =6 \cdot 7 + 7^{2} + 4 \cdot 7^{3} + 5 \cdot 7^{4} + 2 \cdot 7^{5} + 5 \cdot 7^{6} + 3 \cdot 7^{7} + 6 \cdot 7^{8} + 4 \cdot 7^{9} +O(7^10) $$
    
    to deduce the value of   $$\int_b^P \omega \eta = 6 \cdot 7 + 3 \cdot 7^{2} + 6 \cdot 7^{3} + 6 \cdot 7^{4} + 7^{5} + 4 \cdot 7^{6} + 7^{7} + O(7^{9}).$$
    
    Meanwhile, $E$ has the $3$-torsion point $T$ with 
    
    \begin{align*}x(T) &= 3 + 6 \cdot 7 + 6 \cdot 7^{2} + 3 \cdot 7^{3} + 6 \cdot 7^{4} + 5 \cdot 7^{5} + 6 \cdot 7^{6} + 7^{8} + 5 \cdot 7^{9} + O(7^{10})  \\
    y(T) &= 2 + 5 \cdot 7 + 5 \cdot 7^{2} + 3 \cdot 7^{3} + 5 \cdot 7^{4} + 4 \cdot 7^{5} + 6 \cdot 7^{6} + 3 \cdot 7^{7} + 4 \cdot 7^{9} + O(7^{10}).
    \end{align*}

    We compute \begin{align*}\int_b^T \eta &= 1 + 7 + 4 \cdot 7^{2} + 3 \cdot 7^{3} + 3 \cdot 7^{4} + 6 \cdot 7^{5} + 7^{6} + 2 \cdot 7^{7} + 7^{8} + 2 \cdot 7^{9} +O(7^{10})\\
  \int_T^P \omega  &= 6 \cdot 7 + 7^{2} + 4 \cdot 7^{3} + 5 \cdot 7^{4} + 2 \cdot 7^{5} + 5 \cdot 7^{6} + 3 \cdot 7^{7} + 6 \cdot 7^{8} + 4 \cdot 7^{9} +O(7^{10})\\
  \int_T^P \omega\eta &=  3 \cdot 7 + 2 \cdot 7^{2} + 3 \cdot 7^{3} + 3 \cdot 7^{5} + 2 \cdot 7^{6} + 3 \cdot 7^{7} + 6 \cdot 7^{8} + O(7^9)\end{align*}
to deduce
\begin{align*}\int_b^T \omega\eta &=  \int_b^P\omega\eta - \int_b^T\eta\int_T^P \omega - \int_T^P \omega\eta\\
  &=4 \cdot 7 + 7^{3} + 6 \cdot 7^{4} + 5 \cdot 7^{5} + 7^{7} + 2 \cdot 7^{8} + O(7^9). \end{align*}

Meanwhile, the corresponding $3$-torsion point $T_1$ on the minimal
model has coordinates

\begin{align*}x(T_1) &=5 + 3 \cdot 7 + 2 \cdot 7^{2} + 5 \cdot 7^{3} + 2 \cdot 7^{5} + 3 \cdot 7^{6} + 5 \cdot 7^{7} + 7^{8} + 3 \cdot 7^{9} + O(7^{10})  \\
  y(T_1) &= 6 + 5 \cdot 7 + 7^{4} + 5 \cdot 7^{5} + 4 \cdot 7^{6} + 6 \cdot 7^{7} + O(7^{10}).\end{align*}
and we compute directly using the formula that \begin{align*}\int_b^{T_1}\omega'\eta' &= \frac{1}{3}\log(2y(T_1)+a_1 x(T_1) + a_3)) \\
  &= 4 \cdot 7 + 7^{3} + 6 \cdot 7^{4} + 5 \cdot 7^{5} + 7^{7} + 2 \cdot 7^{8} + 7^{9} +O(7^{10}). \end{align*}

\end{example}

\end{document}